\documentclass{amsart}

\usepackage{amssymb, amsmath, amsthm}

\theoremstyle{plain}
\newtheorem{thm}{Theorem}[section]

\newtheorem{corollary}[thm]{Corollary}

\theoremstyle{definition}

\newtheorem{remark}[thm]{Remark}
\newtheorem{example}[thm]{Example}

\newcommand{\R}{\mathbb R}

\newcommand{\N}{\mathbb N}
\newcommand{\DD}{\mathbb D}
\newcommand{\C}{\mathbb C}

\newcommand{\T}{\mathbb T}
\newcommand{\oD}{\overline{\mathbb{D}}}

\newcommand{\om}{\Omega}

\newcommand{\vp}{\varphi}
\usepackage{mathrsfs}

\newcommand{\spn}{{\rm span}}

\title[Mixing Operators with prescribed unimodular eigenvectors]{ {Mixing operators with prescribed  unimodular eigenvalues}}
\author{H.-P. Beise}
\address{Fachbereich Informatik
\newline\indent Hochschule Trier
\newline\indent D-54293 Trier, Germany}
\email{H.Beise@inf.hochschule-trier.de}

\author{L. Frerick}
\address{Fachbereich IV Mathematik
\newline\indent Universit\"at Trier
\newline\indent D-54286 Trier, Germany}
\email{frerick@uni-trier.de} 

\author{J. M\"uller}
\address{Fachbereich IV Mathematik
\newline\indent Universit\"at Trier
\newline\indent D-54286 Trier, Germany}
\email{jmueller@uni-trier.de}

\begin{document}

\begin{abstract}
For arbitrary closed countable subsets $Z$ of the unit circle  examples of topologically mixing operators on Hilbert spaces are given which have a densely spanning set of eigenvectors with unimodular eigenvalues restricted to $Z$. In particular, these operators cannot be ergodic in the Gaussian sense.
\end{abstract}


\date{\today}
\maketitle

\section{Introduction and main result}

The dynamical  behaviour  of linear operators acting on Fr\'echet spaces $X$ has been investigated intensively in the last years.  Recommended  introductions are the textbooks \cite{BayMath} and \cite{GEandPeris}, and also the recent article \cite{GMM}. It turns out that the richness of eigenvectors corresponding to  unimodular eigenvalues strongly influences the metric dynamical properties of linear operators. In particular, a linear operator on a Hilbert space admits a Gaussian invariant measure of full support if and only if it has spanning unimodular eigenvectors  and is ergodic  in the Gaussian sense (that is, ergodic with respect to a Gaussian measure of full support) if and only if it has  perfectly spanning  unimodular eigenvectors. For these and corresponding results we refer in particular to \cite{BG1}, \cite{BG2}, \cite{BayMath2}, \cite{Gr} and again to \cite{GMM}. 

Due to recent deep results of Menet (\cite{Menet}) and Grivaux, Matheron and Menet (\cite{GMM}), in the situation of Hilbert spaces $X$ the picture has become quite complete for the case of chaotic operators, that is, for hypercyclic operators having eigenvectors corresponding to roots of unity (i.e. periodic vectors) which span a dense subspace of $X$. Less is known, however, in the case of absence of periodic or almost periodic vectors (cf. \cite[Section 1.3]{GMM}). 
In \cite[Question 3]{Gr10} (see also \cite[Question 7.7]{GMM}) it is asked if a hypercyclic operator with a densely spanning set of eigenvectors corresponding to rationally independent eigenvalues is already ergodic. 

In this paper, we give examples of topologically mixing operators on Hilbert spaces  which  have a densely spanning set of eigenvectors with unimodular eigenvalues restricted to an arbitrary prescribed closed, countable subset $Z$ of the unit circle $\T$. In particular, such operators cannot be ergodic in the Gaussian sense. Choosing $Z$ as a rationally independent set, Question 3 from \cite{Gr10} can be answered in the negative, at least in the weak form that ergodicity  in the Gaussian sense does not always  follow from the existence of a spanning set of eigenvectors corresponding to a rationally independent set of unimodular eigenvalues.

For an open set $\om$ in the extended plane $\C_\infty$  we denote by $H(\om)$ the Fr{\'e}chet space of functions holomorphic in $\om$ and vanishing at $\infty$ endowed with the topology of locally uniform convergence, where, as usual,  via stereographic projection we identify  $\C_\infty$ and the sphere $\mathbb{S}^2$ endowed with the spherical metric.
We consider Bergman spaces on general open sets $\om \subset \C_\infty$: For $0 \le p<\infty$ let $A^p(\om)$ be the space of all functions $f$ holomorphic in $\om$ that satisfy
\[
||f||_p:=||f||_{\om,p}:=\Big(\int_\om |f|^p \, dm_2\Big)^{1/p} < \infty, 
\]
where $m_2$ denotes the spherical measure on $\C_\infty$.  Then $A^0(\om)=H(\om)$ and   for $p \ge 1$ the spaces $(A^p(\om),||\cdot||_{p})$ are Banach spaces. In the case $p=2$, the norm is induced by the inner  product $
(f,g)\mapsto \int_\om f\overline{g}\, dm_2$.

In the sequel we always consider open sets $\om$ with $0 \in \om$ in which case define  
$T=T_{A^p(\om)}: A^p(\Omega) \to A^p(\Omega)$
by
\[
Tf(z):= (f(z)-f(0))/z \quad (z \not=0), \qquad Tf(0):=f'(0).
\] 
 If $f(z)=\sum_{\nu=0}^\infty a_\nu z^\nu$, then 
\[
Tf(z)=\sum_{\nu=0}^\infty a_{\nu+1} z^\nu
\]
for $|z|$ sufficiently small. 
We call $T$ the Taylor (backward) shift on $A^p(\om)$. 
Writing $S_n f(z):=\sum_{\nu=0}^n a_\nu z^\nu$ for the $n$-th partial sum of the Taylor expansion $\sum_{\nu=0}^\infty a_\nu z^\nu$ of $f$ about $0$, by induction it is easily seen that, for $n \in \N_0$, 
\begin{equation}\label{iterates}
T^{n+1}f(z)=(f-S_{n}f(z))/z^{n+1} \quad (z \not=0) \quad \text{and} \quad T^{n+1}f(0)=a_{n+1}.
\end{equation}
 
In \cite{BMM_X} it is shown that for open sets $\om$ with $0\in \om$ the Taylor shift $T$ is topologically mixing on $H(\om)$  if and only if each connected component of $\C_\infty \setminus \om$  meets $\T$. Results concerning topological and metric dynamics of the Taylor shift on Bergman spaces are proved  in \cite{BM_p}
and \cite{MT}.

We write $M^*:=(\C_\infty \setminus M)^{-1}$ for $M \subset \C_\infty$. Then  $\om^*$ is a compact plane set (note that $0 \in \om$) and the spectrum of $T$ is contained in $\om^*$. In the case $1\le p <2$  is easily seen that $f \in A^p(\om)$ is an eigenfunction for $T$  if and only if, for some $\alpha \in \om^*$, the function $f$ is a scalar multiple of $\gamma_\alpha$, where $\gamma_\alpha$ is defined by 
\[
\gamma_\alpha(z)=1/(1-\alpha z)
\]
 for $z \in \om$ (with $0\cdot \infty:=0$). In this case, $\alpha$ is the corresponding eigenvalue and the spectrum as well as the point spectrum both equal $\om^*$. If $p \ge 2$, then the functions $\gamma_\alpha$ still belong to $A^p(\om)$ for all $\alpha$ in the interior of  $\om^*$, but in general not for $\alpha$ belonging to the boundary of $\om^*$. If $\om$ has small spherical measure near a boundary point $1/\alpha$ of $\om$, it may, however, happen that $\gamma_\alpha$ is again an eigenfunction of $T$ (that is, $\gamma_\alpha$ belongs to $A^p(\om)$). For example, $\gamma_1 \in A^2(\om)$ in the case of the crescent-shaped region $\om=\DD \setminus \{z:|z-1/2|\le 1/2\}$ where $\DD$ denotes the open unit disc in $\C$. This opens up the possibility to  place eigenvalues at certain points of the boundary of $\om^\ast$. A corresponding construction leads to our main result. We write $\mathcal{E}(T)$ for the set of unimodular eigenvalues of $T$, which for $T=T_{A^2(\om)}$ equals the set of $\lambda \in \T$ such that $\gamma_\lambda \in A^2(\om)$. 
 
 \begin{thm} \label{main_result}
Let  $Z\subset \T$ be an infinite closed set. Then there is a open set $\om \subset \DD$ so that the Taylor shift $T=T_{A^2(\om)}$ is topologically mixing,  $\mathcal{E}(T) \subset Z$ and $\{\gamma_\lambda: \lambda \in \mathcal{E}(T)\}$ spans a dense subspace of $A^2(\om)$.
\end{thm}   
 
\begin{remark}
 Since Hilbert spaces are of cotype 2, the  main theorem from \cite{BayMath2} implies that, in the situation of Theorem \ref{main_result}, for countable $Z$ the Taylor shift $T$ is not ergodic in the Gaussian sense. So, as already mentioned above,  Question 3 from \cite{Gr10} can be answered in the negative in so far as the existence of a spanning set of eigenvectors corresponding to a rationally independent set of unimodular eigenvalues does not always imply ergodicity in the Gaussian sense. We are left with the open question whether $T$ is ergodic with respect to some measure of full support or (upper) frequently hypercyclic. 

 If  in the situation of Theorem \ref{main_result} the set $Z$ consists of roots of unity, then the Taylor shift is chaotic with unimodular eigenvalues only in $Z$ and not ergodic in the Gaussian sense. So we have an alternative construction for a chaotic operator on a Hilbert space that is not ergodic in the Gaussian sense. The first construction of such an operator on $\ell_2(\N)$ given by Menet (\cite{Menet}) even leads to an operator which is not (upper) frequently hypercyclic. The approach via Taylor shift is, however, quite different and gives more flexibility in prescribing unimodular eigenvectors.  
\end{remark}

\section{Proof of Theorem \ref{main_result}}

As main tool for the proof of Theorem \ref{main_result} we seek results on rational approximation in the mean. In the sequel, we restrict to open sets $\om$ which are bounded in $\C$ or contain the point $\infty$. 

\begin{remark}\label{Hedberg}
Let $\om$ be an open and bounded set in $\C$ and suppose that each point on the boundary of $\om^\ast$ belongs to the boundary of some component of the interior $(\om^\ast)^\circ$ of $\om^\ast$. Then Theorem 4  from \cite{Hed} implies that the span of $\{\gamma_\alpha: \alpha \in (\om^\ast)^\circ\}$ is dense in $A^2(\om)$. If $\om$ is open in the extended plane with $\infty \in \om$, this is also the case (see \cite[Remark 2.6]{MT}).
\end{remark}

As indicated in the introduction,  the function $\gamma_\alpha$ belongs to $A^2(\om)$ if $\alpha$ belongs to the boundary of $\om^\ast$ and $\om$ has small spherical measure near $1/\alpha$. 
For $k \in \N_0$, $z \in \om$ and $\alpha \in \C$ we write 
\[
\gamma_{\alpha,k}(z):=z^k/(1-\alpha z)^{k+1}=z^k\gamma_\alpha^{k+1}(z).
\]  
We show that in the case of sufficiently small measure near $1/\alpha$ all functions $\gamma_{\alpha,k}$ belong to $A^2(\om)$ and that under appropriate conditions they span a dense subspace of $A^2(\om)$. The approach is strongly influenced by the proof of a result on the completeness of polynomials in $A^2(\om)$ (see \cite[Theorem 12.1]{Mer}, cf.\ also \cite[Chapter I, Section 3]{Gai}). 
   
Let $\delta>0$ and let the "cup" $C_\delta$ be defined as the interior of the convex hull of  $\{t+is(t):-\delta \le t\le\delta\}$, where 
\[
s(t):=\exp(-\exp(1/|t|))
\] 
(with $s(0)=0$). 
With that, we say that two components $A,B$  of an open set $U \subset S:=\R +i(-\pi/2, \pi/2)$ are directly bridged at $w\in S$, if $\omega \in \T$  and $\delta > 0$ exist with $w+\omega C_\delta \subset A$ and $w-\omega C_\delta \subset B$.  If $\varphi:S \to \mathbb{S}^2$ is the 
standard parametrisation of $\mathbb{S}^2 \setminus \{\pm (0,0,1)\}$, that is, 
\[
\varphi(t+is)=(\cos(s)\cos(t),\, \cos(s)\sin(t),\, \sin(s)) 
\]
for $t \in \R$ and $-\pi/2 <s<\pi/2$, we say that two components $C,D$ of an open set $V \subset \C_\infty$ with $0, \infty \not\in \partial V$ are directly bridged, if the corresponding inverse images under $\varphi$ in $S$ are bridged at some $w$. In this case $\zeta=\varphi(w)$ is said to be a bridge point for $(C,D)$. We say that $C,D$ are bridged if finitely many components $C_0, C_1, \ldots, C_m$ exist with $C_0=C$, $C_m=D$ and so that $C_j, C_{j-1}$ are directly bridged. If $\mathcal{C}$ denotes the set of components of $V$, then bridging induces an equivalence relation $\sim$ on $\mathcal{C}$.  If a system $\mathcal{D} \subset \mathcal{C}$ is a complete system of representatives for $\sim$, we briefly say that the system is complete for $V$. 

\begin{thm}\label{Approx}
Let $\om$ be an open set in $\C_\infty$ which is bounded in $\C$ or contains $\infty$ and suppose that each point on the boundary of $\om^\ast$ belongs to the boundary of some component of the interior $(\om^\ast)^\circ$ of $\om^\ast$.  Moreover, suppose $\mathcal{D}$ to be complete for $(\om^\ast)^\circ$ and for  $D \in \mathcal{D}$ let $\alpha_D$ be either in $D$ or a bridge point of $(C,D)$ for some $C \in \mathcal{C}$. Then both
$\{\gamma_\alpha: \alpha \in \bigcup_{D \in \mathcal{D}} D\}$ 
and $\{\gamma_{\alpha_D,k}: D \in \mathcal{D}, \, k \in \N_0\}$ have dense span in $A^2(\om)$.
\end{thm}

\begin{proof}
1. Let $g \in A^2(\om)=A^2(\om)'$. Then the Cauchy transform $Vg:(\om^\ast)^\circ \to \C$  of $g$, defined by 
\[
(Vg)(\alpha)=\int_{\om} \gamma_\alpha (z)\overline{g}(z)\, dm_2 (z) 
\]
for $\alpha \in (\Omega^\ast)^\circ$, is holomorphic with 
\[
(Vg)^{(k)}(\alpha)=k!\int_{\om} \gamma_{\alpha,k} (z)\overline{g}(z)\, dm_2 (z) 
\]
for all $\alpha \in (\Omega^\ast)^\circ$ and $k \in \N_0$. According to the Hahn-Banach theorem and Remark \ref{Hedberg}, it suffices to show that $Vg|_C=0$ for all $C \in \mathcal{C}$ under each of the two conditions stated in the theorem.

The crucial point is that $Vg$ extends at bridge points $\zeta$ to a function which belongs to a quasi-analytic subclass of $C^\infty(I)$ for some line $I$ with $I\setminus \{\zeta\} \subset (\Omega^\ast)^\circ$: Let $\zeta$ be a bridge point for $(C,D)$. In order to reduce notational complexity we assume that $\zeta = 1$ and that $C$ lies in $\DD$  and $D$ outside $\oD$. Considering the fact that $1$ is a bridge point, we may fix  $0<r< 1$ in such a way that the corresponding "cup"-sets $\pm C_r$ satisfy $\varphi^{-1}(\om)\cap \pm C_r=\varnothing$. Let 
\[
I = \varphi(i[-r/2,r/2])
\]
(note that $I$ is a compact interval in $\R$ with $1$ in its interior).
To show that $Vg|_{I \setminus \{1\}}$ extends at $1$ to a $C^\infty$-function on $I$ and  that the extension (again denoted by $Vg$) belongs to a quasi-analytic subclass of $C^\infty(I)$, 
 we estimate the derivatives of $Vg$ on $I\setminus \{1\}$. 
By the Cauchy-Schwarz inequality we have 
\begin{equation}\label{ineq1}
\lvert (Vg)^{(k)}(x)\rvert \leq k! \int_\Omega |\gamma_{x,k}\overline{g}|\, dm_2\le k! \, \lVert g \rVert_2\, \left(\int_{\om} |\gamma_{x,k}|^2 dm_2  \right)^{1/2}
\end{equation}
for $k \in \N_0$ and $1\not=x \in I$.  So it suffices to estimate the latter integrals. We define 
\[
W_{A}(k, x):=\int_{A} |\gamma_{x,k}|^2 dm_2=\int_{A} \frac{|z|^{2k} dm_2(z)}{\lvert 1-xz\rvert ^{2k+2}}
\]
for $k \in \N_0$, $x \in I$ and measurable $A \subset \om$.  With  $Q:=[-r,r]+i[-r,r]$, we have 
\[
\sup_{x \in I}W_{\om\setminus \varphi(Q)}(k,x) =\mathcal{O}(q_1^k)
\]
 for some positive $q_1$.  To estimate $W_{\varphi(Q)\cap \om}$, we observe that the shape of $\om$ in $\varphi(Q)$ allows that with some constant $c>0$ we have $\lvert 1-xz\rvert\geq c \lvert 1-z\rvert$ for all $x\in I$ and all $z\in\om\cap \varphi(Q)$. Thus, for $x \in I$ we obtain (by substituting $u=e^{1/t}$ in the last step) with a positive constant $q_2$
\begin{eqnarray*}
W_{\om\cap \varphi(Q)}(k,x)& \le & q_2^k W_{\om\cap \varphi(Q)}(k,1)\le c \int_{-r}^r \int_{-s(t)}^{s(t)}\frac{\cos s}{\lvert t+is\rvert^{2k+2}}\, ds\, dt \\ \\
& \leq & 4q_2^k \int_{0}^r \frac{s(t)}{t^{2k+2}}\, dt = 4q_2^k \int_{e^{1/r}}^\infty e^{-u}u^{-1}\log^{2k}(u) \, du.
\end{eqnarray*}
For $k$ sufficiently large and $u\ge k^2$ we have $\log^{2k}(u)\le e^{u/2}$. Hence, by splitting up the integral at $k^2$, one can see that 
\[
\int_{e^{1/r}}^\infty e^{-u}u^{-1}\log^{2k}(u) \, du = \mathcal{O}(k^2\log^{2k}(k^2))=\mathcal{O}(5^k\log^{2k}(k)).
\]
Putting together we find that 
\begin{equation}\label{ineq2}
\sup_{x \in I}W_\om(k,x)=\mathcal{O}(q_3^k \log^{2k}(k))
\end{equation}
for some positive constant $q_3$. Since $|\gamma_{x,k}(z)|\le |\gamma_{1,k}(z)|$  for $z \in \vp(Q)$ and $x\in I$, and since there is a constant $c_k$ with $|\gamma_{x,k}(z)|\le c_k$  for $z \not\in \vp(Q)$ and $x \in I$, the $m_2$-integrability of $g$ on $\om$ and of $|\gamma_{1,k}\overline{g}|$ on $\vp(Q)\cap \om$ imply,  by differentiation of parameter integrals, that $Vg$ extends to a function in $C^\infty(I)$ (which we again denote by $Vg$). 
Moreover, combining (\ref{ineq1}) and \eqref{ineq2} we obtain
\[
\sup_{x \in I}\lvert (Vg)^{(k)}(x) \rvert = \mathcal{O}(k! \, \sqrt{q_3}^{k} \log^{k}(k)) .
\]
Hence, the Denjoy-Carleman theorem shows that $Vg$ belongs to a quasi-analytic subclass of $C^\infty(I)$.

2. Suppose that $g \perp \gamma_\alpha$ for all $\alpha \in  \bigcup_{D \in \mathcal{D}} D$, that is $(Vg)|_D=0$ for all  $D \in \mathcal{D}$.  If $\mathcal{D}=\mathcal{C}$ then $Vg=0$.  If  $C \in \mathcal{C}\setminus \mathcal{D}$   
then $C,D$ are bridged for some $D \in \mathcal{D}$. We can assume that $C,D$ are directly bridged with bridge point $\zeta=1$ as in part 1 above. From the assumption $(Vg)|_D=0$ we have $(Vg)^{(k)}(1)=0$ for all $k$, and then also $(Vg)|_C=0$ by quasi-analyticity an the identity theorem for holomorphic functions. 

3. Suppose that $g \perp \{\gamma_{\alpha_D,k}: D \in \mathcal{D},\, k \in \N_0\}$, that is, $\int_\Omega \gamma_{\alpha_D,k}\overline{g}\,dm_2=0$  for all $D \in \mathcal{D}$ and  $k \in \N_0$, and let $C \in \mathcal{C}$. If $C \in \mathcal{D}$ then $\alpha_C \in C$ or $\alpha_C$ is a bridge point of $C,E$ for some $E \in \mathcal{C}$. In both cases, according to part 1 of the proof (with $\alpha_C=1$ without loss of generality) we have $(Vg)|_C=0$.  If $C \not\in \mathcal{D}$, then there is $D \in \mathcal{D}$ so that $C,D$ are bridged. Again, we can assume that $C,D$ are directly bridged with bridge point $\zeta=1$ as above.  If $\alpha_D \in D$, then $(Vg)|_D=0$. Hence $(Vg)^{(k)}(1)=0$  for all $k$ and then $(Vg)|_C=0$. If $\alpha_D$ is a bridge point of $(D,E)$ for some $E \in \mathcal{C}$, then by the same argument as above  we can now conclude $(Vg)|_D=0$.  As in the first case, $(Vg)|_C=0$.   
\end{proof}

\begin{remark} \label{singleD}
If, for some component $D \in \mathcal{C}$, the single set system $\{D\}$ is complete  for $(\om^\ast)^\circ$ and if $\alpha$ is a point in $D$ or a bridge point of $(C,D)$ for some component $C$, then the span of $\{\gamma_{\alpha,k}: k \in \N_0\}$ is dense in $A^2(\om)$. In particular, in case $\alpha=0$ we conclude that the polynomials form a dense set in $A^2(\om)$  (cf.\ Theorem 12.1 in \cite{Mer}, where actually a weaker condition on the sharpness of $\om$ near $1/\alpha$ is proved to be sufficient).   
\end{remark}

From the Godefroy-Shapiro criterion (see e.g.\  \cite[Theorem 3.1]{GEandPeris}) and Theorem \ref{Approx} we obtain

\begin{corollary} \label{mixing}
Let $\om$ be an open set in $\C_\infty$ which is bounded in $\C$ or contains $\infty$ and suppose that each point in the boundary of $\om^\ast$ belongs to the boundary of some component of  $(\om^\ast)^\circ$. If complete systems $\mathcal{D}$ and $\mathcal{D}'$ for $(\om^\ast)^\circ$ exist with $D \subset \DD$ for all $D\in \mathcal{D}$ and $D' \subset \C_\infty\setminus \overline{\DD}$ for all  $D'\in \mathcal{D}'$, then $T_{A^2(\om)}$ is topologically mixing. 
\end{corollary}

\begin{example}\label{ex}
Let  $\om=\C_\infty\setminus (\overline{U}\cup\overline{G})=((1/\overline{U})\cup (1/ \overline{G}))^\ast$, where $G, 1/U\subset\DD$ are domains with $0 \not\in \overline{G}$ and so that $U,G$ (or, equivalently, $1/G, 1/U$) are bridged. In this case, $\mathcal{D}=\{1/U\}$ and $\mathcal{D}'=\{1/G\}$ are complete systems for $(\om^\ast)^\circ$. 
By Corollary \ref{mixing}, the Taylor shift $T_{A^2(\om)}$ is mixing. If $0 \in 1/U$ (e.g.\ for $1/U=\DD$, in which case $\om=\DD \setminus \overline{G}$), the polynomials are dense in $A^2(\om)$ and the Taylor shift on $A^2(\DD)$ is a quasi-factor of the Taylor shift on $A^2(\om)$.
\end{example}
%


\begin{thm} \label{simple poles}
For each infinite closed set $Z\subset \T$ a domain $G \subset \DD$ bridged to $\C_\infty \setminus \overline{\DD}$ exists with  $\overline{G} \cap \T\subset Z$  and so that the span of $\{\gamma_{\zeta}: \zeta \in \overline{G} \cap \T\}$ is dense in $A^2(\DD\setminus \overline{G})$.
\end{thm}

\begin{proof} 
Let
\[
U_r:=\{z \in \C :|z-1| < r\}\cap\DD.
\]

In our construction we will cut out sets from $\DD$, essentially of the form $\zeta \varphi(C_\delta) $, that are "flat" at points $\zeta \in \T$ in such a way that these $\zeta$ become bridge points to $\C_\infty \setminus \overline{\DD}$. We have to control the integrals of $\gamma_\zeta$ over the remaining areas next to such a bridge point. 
To this end, throughout the proof we fix a $\delta>0$ and a corresponding $r_0>0$ that is sufficiently small such that for all $\zeta\in \T\cap \overline{U_{r_0}}$ the set $ U_{r_0}\setminus \zeta\varphi(C_{\delta})$ consists of the (one or two) component(s) bounded only by $\{\varphi(t+is(t)): t\in [-\delta, \delta]\}$ and $\overline{U_{r_0}}\cap \T$, i.e. for every $z$ in this set we find $t'\in [-\delta, \delta]\setminus\{0\}$ and $\zeta'\in \overline{U_{r_0}}\cap \T$ such that $\varphi(t'+is(t'))$ and $\zeta'$ define a line through the origin and the line segment between those points  contains $z$ and does not intersect with $\varphi(C_{\delta})$.
The definition of $C_{\delta}$ is extended to $C_{\delta,\rho}$, with $\rho>0$, defined as the interior of
\[\text{conv}\{t+i\,ds(t):-\delta \le t\le\delta, \rho \le d\le 1\},\]
with $\rm conv$ denoting the convex hull. Together, this allows us  to decrease for $\zeta\in \T\cap \overline{U_{r_0}}$ the area of $U_r\setminus \zeta\varphi(C_{\delta,\rho})$ to become arbitrarily close to zero by decreasing $\rho$. 

We may suppose that $Z$ has an accumulation at $1$. 
In the initial step of our construction we set  $Z_0:=\varnothing$ and $G_0=\varphi(C_{\delta,1})$.
This implies $\gamma_{1,k}\in  A^2(\DD\setminus \overline{G_0})$ for all $k$. 
For a subsequent $n\in\N$, the iterative procedure goes as follows.  We choose $0<r_n<r_{n-1}$ such that 
\begin{equation}\label{intGammaSmall}
\int_{U_{r_n}\setminus \overline{G_{n-1}}}|\gamma_{1,j}|^2 \, dm_2 <1/n
\end{equation}
for $j=1,...,n$. By a variant of Runge's theorem on approximation by rational functions with first order poles, as it can be found in \cite{luecking2012complex}, we find a finite set
\[Z_n\subset \overline{U_{r_n/2}}\cap \left(Z\setminus (\bigcup_{j=0}^{n-1}Z_j\cup \{1\})\right)\]
and an $R_j=R_{n,j}\in \spn\{\gamma_\zeta: \zeta \in Z_n\}$ such that 
\begin{equation}\label{distRGammaSmall}
\max_{z \in \overline{\DD} \setminus U_{r_n}} |\gamma_{1,j}(z)- R_j(z)|<1/n
\end{equation}
for $j=1,...,n$.
Then we set
\[G_n:=\bigcup_{\zeta\in Z_n}\zeta \varphi(C_{\delta,\rho_n})\]
where $\rho_n>0$ is chosen in a way that the area of $U_{r_n}\setminus \overline{G_n}$ is sufficiently small to yield
\begin{equation}\label{intSmall}
\int_{U_{r_n}\setminus \overline{G_n}}|R_j|^2 \, dm_2 <1/n.
\end{equation}
This implies $\gamma_{\zeta}\in A^2(\DD\setminus \overline{G_n})$ for all $\zeta\in Z_n$ and \eqref{intGammaSmall}, (\ref{distRGammaSmall}), (\ref{intSmall}) give
\[\int_{\DD\setminus \overline{G_n}}|R_j(z)-\gamma_{1,j}(z)|^2 \, dm_2(z) <3/n\]
for $j=1,...,n$.
For $G:=\bigcup_{n \in \N} G_n$ we finally have $\gamma_\zeta\in A^2(\DD\setminus \overline{G})$ for all $\zeta\in\T\cap \overline{G}$, and for every non negative integer $k$ the function $\gamma_{1,k}$ belongs to the closure of the span of $\{\gamma_{\zeta}: \zeta\in \overline{G} \cap \T\}$.  According to Example \ref{ex} with $1/U=\DD$, the latter implies the denseness of $\spn\{\gamma_\zeta: \zeta \in \overline{G} \cap \T\}$ in $A^2(\DD\setminus \overline{G})$.

\end{proof}
\begin{proof}[Proof of Theorem \ref{main_result}]
Let $\om=\DD \setminus \overline{G}$ where $G$ is as in Theorem \ref{simple poles}. Then $\overline{G} \cap \T \subset \mathcal{E}(T)$. 
According to Example \ref{ex} and Theorem \ref{simple poles}, $T$ is topologically mixing. Since each point in $\T \setminus \overline{G}$ is an interior point of $\overline{\DD}$, no point in  $\T \setminus \overline{G}$ belongs to the point spectrum, that is $\mathcal{E}(T) \subset Z$.
\end{proof}  
 
\begin{remark}  
1. By deleting sufficiently small parts from $G$ it is possible to modify $G$ to an open set $W$ in such a way that $\om=\DD \setminus \overline{W}$ is connected and the Taylor shift $T_{A^2(\om)}$ satisfies the same conditions as in Theorem \ref{main_result}. 

2. The statement (and proof) of Theorem \ref{simple poles} can be modified in such a way that $\om^\ast \cap \T \subset Z$, i.e.\ the spectrum intersects $\T$ only in $Z$ (cf.\ \cite[Question 3]{Gr10}).  
\end{remark}



\end{document}